\documentclass{article}
\usepackage{url}
\usepackage{amsmath}
\usepackage{amsthm}
\usepackage{color,graphicx,fancybox,texdraw}
\usepackage{amssymb}
\usepackage[all]{xy}

\theoremstyle{plain}
\newtheorem{theorem}{Theorem}
\newtheorem{lemma}[theorem]{Lemma}
\newtheorem{proposition}[theorem]{Proposition}

\theoremstyle{definition}

\newtheorem{example}[theorem]{Example}

\numberwithin{equation}{section}
\numberwithin{theorem}{section}

\newcommand{\id}[1]{\left\langle#1\right\rangle}

\newcommand{\F}{\mathbb{F}}

\title{Minimal Reversible Nonsymmetric Rings}
\author{Steve Szabo\\Department of Mathematics and Statistic\\Eastern Kentucky University\\Richmond, KY 40475\\steve.szabo@eku.edu}
\date{}

\begin{document}
\maketitle
\begin{abstract}
Marks showed that $\F_2Q_8$, the $\F_2$ group algebra over the quaternion group, is a reversible nonsymmetric ring, then questioned whether or not this ring is minimal with respect to cardinality. In this work, it is shown that the cardinality of a minimal reversible nonsymmetric ring is indeed 256. Furthermore, it is shown that although $\F_2Q_8$ is a duo ring, there are also examples of minimal reversible nonsymmetric rings which are nonduo.
\end{abstract}
\noindent keywords: finite rings, reversible rings, symmetric rings, minimal rings, 2-primal rings

\section{Introduction and Overview}
In \cite{cohn_1999}, Cohn introduced reversible rings. A ring $R$ is reversible if for any $a,b\in R$, $ab=0$ implies $ba=0$. Since Cohn's introduction of reversible rings there has been many works connecting them to other ring properties (see \cite{gutan_2004,kim_2011,kim_2003,marks_2002}). In \cite{lambek_1971}, Lambek studied what are now called symmetric rings. A ring $R$ is symmetric if for any $a,b,c\in R$, $abc=0$ implies $bac=0$. Although Lambek's introduction was some time ago, more recently, the connections between symmetric rings and rings with related properties have been of some interest. In particular, Marks, in \cite{marks_2002}, studied the relationship between reversible and symmetric rings with and without identity. He showed that when considering rings without identity, these two properties are independent. However, it is easy to see that a symmetric ring with identity is reversible. He went on to show that $\F_2Q_8$, the $\F_2$ group algebra over the quaternion group, is a reversible nonsymmetric ring. He then asked whether or not a minimal reversible nonsymmetric ring with respect to cardinality is of order 256. The sole purpose of this paper is to answer this question.

A ring is called {\it right (resp. left) duo} if every right (resp. left) ideal is two-sided. Marks showed that this property even for rings with identity is independent of being reversible nonsymmetric by proving that $\F_2Q_8$ is right duo as well as showing the ring
\[
{k\id{x,y,z}\over\id{x^2,y^2,z^2,xyz,yzx,zxy}}
\]
for a field $k$ is reversible nonsymmetric and neither left nor right duo. This begs the question, ``Does minimality of reversible nonsymmetric rings depend on the duo property?'' With $k=\F_2$, this ring is of order $8192$. We have found a ring of order 256,
\[
\frac{\F_2\id{u,v}}{\id{u^3,v^3,u^2+v^2+vu,vu^2+uvu+vuv}}
\]
which is reversible nonsymmetric and neither left nor right duo (see Example \ref{ex_232}). In analyzing indecomposable reversible nonsymmetric rings, which we know must be local (see Proposition \ref{prop1}), for minimality the nilpotency of the jacobson radical must be at least 4 (Proposition \ref{prop_j3}) and at most 6 (Proposition \ref{prop_21s}). Examples having nilpotency 4 or 5 are given in Section \ref{sect_examples}. As of yet we have not identified an example with nilpotency 6.

There has been some work on specific ring types and minimality. In \cite{xue_1992}, Xue showed that a minimal noncommutative duo ring is of order 16 and that there are only three such rings. In \cite{kim_2011}, Kim and Lee showed that a minimal reversible noncommutative ring is of order 16. In another paper \cite{xu_1998}, Xu and Xue showed that a minimal zero-insertive ring (also known as semicommutative or having S~I condition) also has order 16 and that there are only 5 such rings. From \cite{derr_1994} we know there are only 13 noncommutative rings of order 16. From their list and the fact that there is only one noncommutative ring or order 8, with some effort, the minimal rings just mentioned could be found. Such work becomes difficult when rings of larger orders are involve. With the classification of rings of order $p^5$ in \cite{corbas_2000} and \cite{corbas_2000_2}, we see that such work becomes cumbersome as the classification of finite rings is difficult and that even for small orders, there is an abundance of rings. Hence, the techniques used in the papers mentioned are important to find minimal rings with various properties.

The connections between various ring properties related to commutativity such as reduced, symmetric, reversible, semicommutative and 2-primal are studied in \cite{marks_2003} where the strict inclusion of these ring classes is presented. Adding the fact that a ring is reversible if and only if it is semicommutative and reflexive to their ring class inclusions we have the following diagram.
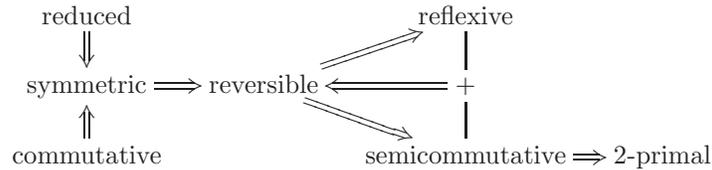
\begin{center}
~\xymatrixrowsep{1pc}\xymatrixcolsep{1pc}\xymatrix{
\textrm{reduced}\ar@{=>}[d]&&\textrm{reflexive}\ar@{<=}[dl]\ar@{-}[d]\\
\textrm{symmetric}\ar@{=>}[r]&\textrm{reversible}&+\ar@{=>}[l]\\
\textrm{commutative}\ar@{=>}[u]&&\textrm{semicommutative}\ar@{<=}[ul]\ar@{=>}[r]\ar@{-}[u]&\textrm{2-primal}\\
}
\end{center}
Having minimal examples showing these strict inclusions can be instructive when studying them and our result here provides such an example. In a companion paper, \cite{szabo_2017_2}, we list minimal rings of these classes. Interestingly, many of the examples come from the set of noncommutative rings of order 16. It is shown there for instance that $\frac{\F_4[x;\sigma]}{\id{x^2}}$ where $\sigma$ is the Frobenius automorphism on $\F_4$, is the minimal noncommutative nonreduced symmetric ring and that $\frac{\F_2\id{u,v}}{\id{u^2,v^2,uv}}$ is a minimal semicommutative nonreversible ring. Both are rings of order 16. As we will see here, minimal reversible nonsymmetric rings are much larger than minimal rings of related types. Notice that reversible nonsymmetric rings in some sense, live between noncommutative nonreduced symmetric rings and semicommutative nonreversible rings. It is curious then that minimal nonsymmetric rings are significantly larger than these other two types.

In Section \ref{sect_prelim} preliminaries are presented. Section \ref{sect_main} contains the main result of the paper, namely that a minimal reversible nonsymmetric ring has order 256. Finally, in Section \ref{sect_examples} examples are provided.

\section{Preliminaries}
\label{sect_prelim}
We assume all rings have unity and that rings are finite with the exception of a polynomial ring. Given a ring $R$, \textit{$J(R)$} is the Jacobson radical of $R$ and $S(R)$ is the socle of $R$. A ring $R$ is {\it symmetric} if for all $a,b,c\in R$, $abc=0$ implies $bac=0$. A ring $R$ is {\it reversible} if for all $a,b\in R$, $ab=0$ implies $ba=0$. The following is a well-known result.

\begin{lemma}[see \cite{jacobson_1964}]
\label{lemma_vspace}
Let $R$ be a finite local ring and $F=R/J(R)$. Then $F$ is a field and $J(R)^k/J(R)^{k+1}$ is a finite dimensional $F$-vector space.
\end{lemma}

The following results on reversible rings are needed throughout.

\begin{lemma}
\label{lemma1}
Let $R$ be a ring which contains an idempotent $e$ such that $eR(1-e)\neq0$. Then $R$ is not reversible.
\end{lemma}
\begin{proof}
Let $a\in eR(1-e)\ne0$. Then $ea=a\neq 0$ and $ae=0$ showing $R$ is not reversible.
\end{proof}

\begin{proposition}
\label{prop1}
Let $R$ be a reversible ring. Then $R$ is local if and only if it is indecomposable.
\end{proposition}
\begin{proof}
Local rings are indecomposable rings. Indecomposable rings have no central idempotents other than 1. The result follows from Lemma \ref{lemma1}.
\end{proof}

\begin{proposition}
\label{prop_j3}
Let $R$ be a local reversible ring. If $J(R)^3=0$ then $R$ is symmetric.
\end{proposition}
\begin{proof}
Assume $J(R)^3=0$. Let $a,b,c\in R\setminus 0$ and assume $abc=0$. Since $R$ is local, at least two of $a$, $b$ and $c$ are in $J(R)$. Assume $b$ is a unit. Since $R$ is reversible, we have that $cab=0$, $ca=0$, $ac=0$ and $bac=0$. This logic also shows if either $a$ or $c$ is a unit, $bac=0$. Now assume $a,b,c\in J(R)$. Since $J(R)^3=0$, $bac=0$. Hence, $R$ is symmetric.
\end{proof}

Gutan in \cite{gutan_2004} characterized reversible group rings, provided examples of reversible nonsymmetric group rings and  showed that $\F_2Q_8$ is a minimal reversible nonsymmetric group ring. This result is easy to see in light of Lemma \ref{lemma1}. The only noncommutative group ring smaller than $\F_2Q_8$ is $\F_2S_3$. Since $\F_2S_3$ has a subring isomorphic to $M_2(\F_2)$, by Lemma \ref{lemma1}, it is not reversible. This verifies Gutan's result. Gutan left open the question of $\F_2Q_8$ being a minimal reversible nonsymmetric ring or not.

\section{Minimal Reversible Nonsymmetric Rings}
\label{sect_main}
\begin{theorem}
\label{theo_main}
A minimal reversible nonsymmetric ring has order 256.
\end{theorem}

The proof of Theorem \ref{theo_main} will be given at the end of this section. As Marks showed that $\F_2Q_8$ is a reversible nonsymmetric ring of order 256, we show that no ring of order less than 256 is reversible but nonsymmetric. Since the direct sum of a set of reversible nonsymmetric rings is reversible nonsymmetric and by proposition \ref{prop1}, a reversible ring is indecomposable if and only if it is local, we need only consider local  reversible rings.

Throughout this section, let $R$ be a finite local reversible ring of order less than 256, $J=J(R)$, $S=S(R)$ and $F=R/J$. Since $R$ is finite and local, $F$ is a field. From Lemma \ref{lemma_vspace} we know $J^i/J^{i+1}$ is an $F$-vector space. Let $d_i=\dim_F(J^i/J^{i+1})$ and $D=(d_1,d_2,\dots,d_K)$ where $K$ is such that $d_K\neq 0$ and $d_{K+1}=0$. If $J^3=0$, by Proposition \ref{prop_j3} $R$ is symmetric. So, assume $J^3\neq 0$. Then $|R|\geq |F|^4$. If $|F|\geq 4$, $|R|\geq 256$ but we have assumed $|R|<256$. Therefore, $F\in\{\F_2,\F_3\}$. Now, $|R|=p^m$ for a prime $p$ and some positive integer $m$. Since $|R|<256$, $|R|\in\{2^4,2^5,2^6,2^7,3^4,3^5\}$.

\begin{proposition}
\label{prop_21s}
Assume $d_2=1$. If $d_1\leq 2$ or $F\cong\F_2$ then $R$ is commutative.
\end{proposition}
\begin{proof}
From the assumption, $J=Fa_1+Fa_2+\dots+Fa_n+J^2$ for $a_1,\dots,a_n\in J\setminus J^2$, $d_2=\dots=d_K=1$, $J^{K-1}=Ft+S(R)$ for some $t\in J^{K-1}$ and $S(R)=JJ^{K-1}=Fa_1t+Fa_2t+\dots+Fa_nt$. Since $d_K=1$, we may assume $S(R)=Fa_1t$ after renumbering. For some $\lambda_i\in F$, $a_it=\lambda_ia_1t$ and so $(a_i-\lambda_ia_1)t=0$. Let $a_i'=(a_i-\lambda_ia_1)$ for $i>1$. Then $J=Fa_1+Fa_2'+\dots+Fa_n'+J^2$ and since $R$ is reversible, $a_i't=0=ta_i'$ for $i>1$. For notational purposes replace $a_i$ with $a_i'$.

There must be a product of $K-1$ elements from $\{a_1,a_2,\dots,a_n\}$ in $J^{K-1}\setminus S(R)$. Let $x_1x_2\dots x_{K-1}$ be such a product. Then $a_1x_1\dots x_{K-1}\neq 0$ since $a_1$ does not annihilate $J^{K-1}$. By reversibility, any cyclic shift of the product is also nonzero. Since $a_i$ for $i>1$  annihilates $J^{K-1}$, $x_i=a_1$ for $1\leq i\leq K-1$. Hence, $a_1^K\neq 0$ and $J^j/J^{j+1}=Fa_1^j+J^{j+1}$ for $2\leq j\leq K$.

Let $2\leq i\leq n$, $a_ia_1=\sum_{j=2}^K\alpha_j a_1^j$ and $a_1a_i=\sum_{j=2}^K\beta_j a_1^J$ for some $\alpha_j,\beta_j\in F$. Since $F$ is a prime field,
\[
\sum_{j=2}^K\beta_j a_1^{j+1}=a_1a_ia_1=\sum_{j=2}^K\alpha_j a_1^{j+1}\\
\]
So, $\alpha_j=\beta_j$ showing $a_ia_1=a_1a_i$. Then $a_1\in Z(R)$ and if $n=2$ then $R$ is commutative. So, now we further assume $F\cong\F_2$.

Now let $2\leq i,l\leq n$. If $a_la_1=0$ then $a_ia_la_1=0$. This implies that $a_ia_l\in S$. Since $F\cong\F_2$, by reversibility, $a_ia_l=a_la_i$. So, assume $a_la_1\neq0$. Then $a_ia_l=\sum_{j=2}^K\gamma_j a_1^j$, $a_la_i=\sum_{j=2}^K\delta_j a_1^j$ and $a_la_1=a_1a_l=\sum_{k=2}^K\sigma_k a_1^k$ for some $\gamma_j,\delta_j,\sigma\in F$. Since $F\cong\F_2$ and $a_1\in Z(R)$,
\[
\sum_{j,k=2}^K\delta_j\sigma_k a_1^{j+k-1}=a_la_ia_l=\sum_{j,k=2}^K\gamma_j\sigma_k a_1^{j+k-1}
\]
From this, since $a_la_1\neq0$, it can be shown that $\gamma_j=\delta_j$. In any case, $a_ia_l=a_la_i$ and $R$ is commutative.
\end{proof}

If $m\in\{4,5\}$ the only possibilities for $D$ are $(1,1,1)$, $(1,1,1,1)$ or $(2,1,1)$. Proposition \ref{prop_21s} shows in these cases $R$ is symmetric. At this point, if $R$ is nonsymmetric then $F\cong\F_2$ and $|R|\in\{2^6,2^7\}$. In light of Proposition \ref{prop_21s} the only possibilities for $D$ would be $(2,2,1)$, $(2,2,1,1)$, $(2,2,2)$, $(2,3,1)$ or $(3,2,1)$. The next proposition rules out $D\in\{(2,2,1),(2,2,2),(2,3,1)\}$.

\begin{proposition}
\label{prop_23}
If $R$ is nonsymmetric, $F\cong\F_2$ and $D=(2,l,n)$ then $l=3$ and $n\geq 2$.
\end{proposition}
\begin{proof}
Assume $R$ is non-symmetric, $F\cong\F_2$ and $D=(2,l,n)$ for some $l$ and $n$. Then there exists $a,b,c\in J\setminus J^2$ such that $abc=0$ and $bac\neq 0$. If $a+J^2=b+J^2$, since $J^4=0$, $bac=a^2c=abc=0$. So, $a+J^2\neq b+J^2$. It can be similarly shown that $a+J^2\neq c+J^2$ and $b+J^2\neq c+J^2$. So, $a+J^2$ and $b+J^2$ are linearly independent in $J/J^2$ showing $J=Fa+Fb+J^2$ and $c=a+b+t$ for some $t\in J^2$. Since $R$ is reversible, $0=abc=aba+ab^2$, $0=cab=a^2b+bab$ and $0=bca=ba^2+b^2a$. So, $aba=ab^2$, $a^2b=bab$ and $ba^2=b^2a$. Again by reversibility, if any degree two monomial involving  $a$ and $b$ is in $J^3$, all degree 3 monomials involving $a$ and $b$ other than $a^3$ and $b^3$ are 0. This would mean $bac=0$. So, $a^2,ab,ba,b^2\in J^2\setminus J^3$. Assume
\[
\alpha a^2+\beta ab+\gamma ba+\delta b^2\in J^3
\]
for some $\alpha,\beta,\gamma,\delta\in F$. Considering left and right multiples of $a$ and $b$ we have that
\begin{eqnarray*}
0&=&\alpha a^3+\beta a^2b+\gamma aba+\delta ab^2\\
0&=&\alpha a^3+\beta aba+\gamma ba^2+\delta b^2a\\
0&=&\alpha ba^2+\beta bab+\gamma b^2a+\delta b^3\\
0&=&\alpha a^2b+\beta ab^2+\gamma bab+\delta b^3
\end{eqnarray*}
and with $aba=ab^2$, $a^2b=bab$ and $ba^2=b^2a$ this becomes
\begin{eqnarray*}
0&=&\alpha a^3+\beta bab+\gamma aba+\delta aba\\
0&=&\alpha a^3+\beta aba+\gamma ba^2+\delta ba^2\\
0&=&\alpha ba^2+\beta bab+\gamma ba^2+\delta b^3\\
0&=&\alpha bab+\beta aba+\gamma bab+\delta b^3.
\end{eqnarray*}
Adding the first and second equations,
\[
0=(\beta+\gamma+\delta)aba+\beta bab+(\gamma+\delta) ba^2.
\]
If $\beta=1$ then either $aba=bab$ or $ba^2=bab$. In either case, $bac=0$. So, $\beta=0$. Then
\[
0=(\gamma+\delta)aba+(\gamma+\delta)ba^2
\]
and adding the third and fourth equations above,
\[
0=(\alpha+\gamma)bab+(\alpha+\gamma)ba^2.
\]
If $\alpha\neq\gamma$, $bab=ba^2$. Similarly, if $\delta\neq\gamma$, $aba=ba^2$. In either case, $bac=0$. So, $\alpha=\gamma=\delta$. Then $\alpha=\gamma=\delta=1$ gives a non-zero solution and $a^2+J^3$, $ba+J^3$ and $b^2+J^3$ are linearly dependent in $J^2/J^3$. From these calculations we see then that $ab+J^3$, $ba+J^3$ and $b^2+J^3$ are linearly independent in $J^2/J^3$. We know that $J^2=Fa^2+Fab+Fba+Fb^2+J^3$ so $l=3$.

It was shown earlier that $a^2b=bab$ and $ba^2=b^2a$. So, by reversibility, either $a^2b=bab=b^2a=ba^2=0$ or $a^2b=bab\neq 0$ and $ba^2=b^2a\neq 0$. Since $ba^2+bab=bac\neq0$, $ba^2$ and $bab$ are both nonzero and hence linearly independent in $J^3$. So, $n\geq 2$.
\end{proof}

In Section \ref{sect_examples}, a few examples of reversible nonsymmetric rings are discussed. The proof of Proposition \ref{prop_23} led to Example \ref{ex_232} which turns out to be a minimal reversible nonsymmetric non-duo ring.

The next proposition takes care of the case when $F\cong\F_2$ and $D=(3,2,1)$.

\begin{proposition}
\label{prop_321}
If $F\cong\F_2$ and $D=(n,2,1)$ for some $n\geq 2$ then $R$ is symmetric.
\end{proposition}
\begin{proof}
Assume $F\cong\F_2$, $D=(n,2,1)$ and $R$ is non-symmetric. Since $J^3\neq0$ and $J^4=0$, there exists $a,b,c\in J\setminus J^2$ such that $abc=0$ and $bac\neq 0$. Since $R$ is reversible we can deduce the following: $S=\{0,bac\}$, $acb=cba=bac\neq 0$,  $abc=cab=abc=0$, $ac,cb\in J^2\setminus J^3$, $ca,bc\neq 0$,  $a^2b=aba=ba^2$, $b^2a=bab=ab^2$ and $b^2c=bcb=cb^2$. We know that $J^2=Fab+Fba+J^3$ because otherwise $abc=bac$. Since $1-a$ is a unit, $(1-a)cb\neq 0$ which shows $c(1-a)b\neq 0$ and $bc\neq bac$. Then $bc\in J^2\setminus J^3$.

Now, $cb=\alpha ab+\beta ba+s$ for some $\alpha,\beta\in F$ not both 0 and $s\in J^3$. Then $0\neq acb=\alpha a^2b+\beta aba=(\alpha+\beta)a^2b$ so $a^2b\neq 0$ and $\alpha\neq\beta$. In any case, 
\[
b^2c=cb^2=\alpha ab^2+\beta bab=(\alpha+\beta) b^2a=b^2a.
\]
Next, $ac=\epsilon ab+\zeta ba+u$ for some $\epsilon,\zeta\in F$ not both 0 and $u\in J^3$. Then $0\neq acb=\epsilon ab^2+\zeta bab=(\epsilon+\zeta)b^2a$ so $b^2a\neq 0$. Also, $bc=\gamma ab+\delta ba+t$ for some $\gamma,\delta\in F$ not both 0 and $t\in J^3$. Then $0=abc=\gamma a^2b+\delta aba=(\gamma+\delta)a^2b$. Since $a^2b\neq 0$, $\gamma=\delta=1$ showing $bc=ab+ba+t$. Then
\[
0\neq b^2a=b^2c=bab+b^2a=0.
\]
Hence, if $F\cong\F_2$ and $D=(n,2,1)$ then $R$ is symmetric.
\end{proof}

We have one final case to rule out, that being when $|R|=2^7$ with $D=(2,2,1,1)$.

\begin{proposition}
\label{prop_2211}
If $F\cong\F_2$ and $D=(2,2,1,1)$ then $R$ is symmetric.
\end{proposition}
\begin{proof}
Assume $F\cong\F_2$ and $D=(2,2,1,1)$. Then $J=Fa_1+Fa_2+J^2$ for $a_1,a_2\in J\setminus J^2$, $J^3=Ft+J^4$ for some $t\in J^3$ and $J^4=JJ^3=Fa_1t+Fa_2t$. Since $d_4=1$, we may assume $J^4=Fa_1t$ after renumbering. For some $\lambda\in F$, $a_2t=\lambda a_1t$ and so $(a_2-\lambda a_1)t=0$. Replace $a_2$ with $a_2-\lambda a_1$. Since $R$ is reversible, $a_2t=0=ta_2$.

There must be a product of $3$ elements from $\{a_1,a_2\}$ in $J^3\setminus J^4$. Let $x_1x_2x_3$ be such a product. Then $a_1x_1x_2x_3\neq 0$ since $a_1$ does not annihilate $J^3$. By reversibility, any cyclic shift of the product is also nonzero. Since $a_2$  annihilates $J^3$, $x_1=x_2=x_3=a_1$. Hence, $a_1^4\neq 0$,  $J^3/J^{4}=Fa_1^3+J^4$ and $J^4=Fa_1^4$. Further more there exists $u\in\{a_1a_2,a_2a_1,a_2^2\}$ such that $J^2/J^3=Fa_1^2+Fu+J^3$. Note, $a_1u,a_2u,ua_1,ua_2\in J^4$ since $a_1ua_1$, $a_2ua_1$, $ua_1a_1$ and $ua_2a_1$ are all 0.

Let $a,b,c\in R$ such that $abc=0$. If $a$, $b$ or $c$ is a unit, $bac=0$. So assume $a,b,c\in J$. Then $a=\alpha_1a_1+\beta_1a_2+\gamma_1a_1^2+\delta_1u+t_1$, $b=\alpha_2a_1+\beta_2a_2+\gamma_2a_1^2+\delta_2u+t_2$ and $c=\alpha_3a_1+\beta_3a_2+\gamma_3a_1^2+\delta_3u+t_3$ for $\alpha_i,\beta_i,\gamma_i,\delta_i\in F$ and $t_i\in J^3$. First, if $\alpha_1=\alpha_2=\alpha_3=1$ then $abc\neq 0$. So, by reversibility we may assume $\alpha_3=0$. If $\beta_3=0$, $bac=0$. So assume $\beta_3=1$. Since $R$ is reversible and $a_1a_2^2\in J^4$, $a_1a_2^2=a_2a_1a_2$. Then

\begin{eqnarray*}
bac=bac-abc&=&(\alpha_2\beta_1-\alpha_1\beta_2)a_1a_2^2+(\alpha_1\beta_2-\alpha_2\beta_1)a_2a_1a_2\\
&=&(\alpha_2\beta_1-\alpha_1\beta_2+\alpha_1\beta_2-\alpha_2\beta_1)a_1a_2^2\\
&=&0
\end{eqnarray*}
Hence, $R$ is symmetric.
\end{proof}

We can now prove Theorem \ref{theo_main}.

\begin{proof}[Proof to Theorem \ref{theo_main}]
In \cite{marks_2002} it was shown that $\F_2Q_8$ is reversible nonsymmetric ring of order 256. Let $T$ be a local reversible ring of order less than 256. If $J(T)^3=0$, by Lemma \ref{prop_j3} $T$ is symmetric. So, assume $J(T)^3\neq0$. By Lemma \ref{lemma_vspace},  $T/J(T)$ is a field and $|T|=\left|T/J(T)\right|^l$ for some $l$. This implies $|T/J(T)|\in\{2,3\}$ and $4\leq l\leq 7$. This is precisely what was assumed about $R$ in this section. We reiterate: $R$ is a finite local reversible ring where $|R|=p^m\in\{2^4,2^5,2^6,2^7,3^4,3^5\}$, $J(R)^3\neq 0$, $J=J(R)$, $F=R/J$, $F$ is a field, $|F|\in\{2,3\}$, $d_i=\dim_F(J^i/J^{i+1})$, $D=(d_1,d_2,\dots,d_K)$ where $K$ is such that $d_K\neq 0$ and $d_{K+1}=0$.

If $m=4$, $D=(1,1,1)$ and if $m=5$, $D\in\{(1,1,1,1),(2,1,1)\}$ then, by Proposition \ref{prop_21s}, $R$ is symmetric. If $m\geq 6$, then $|F|=2$. If $m=6$,
\[
D\in\{(1,1,1,1,1),(2,1,1,1),(2,2,1),(3,1,1)\}
\]
and if $m=7$,
\[
D\in\{(1,1,1,1,1,1),(2,1,1,1,1),(2,2,1,1),(2,2,2),(2,3,1),(3,2,1),(4,1,1)\}.
\]
If
\[
D\in\{(1,1,1,1,1),(2,1,1,1),(3,1,1),(1,1,1,1,1,1),(2,1,1,1,1),(4,1,1)\},
\]
by Proposition \ref{prop_21s} $R$ is symmetric.
If
\[
D\in\{(2,2,1),(2,2,2),(2,3,1)\},
\]
by Proposition \ref{prop_23} $R$ is symmetric. If $D=(3,2,1)$, by Proposition \ref{prop_321} $R$ is symmetric. If $D=(2,2,1,1)$, by Proposition \ref{prop_2211} $R$ is symmetric. So, a local reversible ring of order less than 256 is symmetric. Again, for reversible rings local and indecomposable are equivalent conditions. In a ring direct sum, if a direct summand is nonsymmetric, the ring itself is nonsymmetric. Hence, a minimal reversible nonsymmetric ring is of order 256.
\end{proof}

\section{Examples of Minimal Reversible Nonsymmetric Rings}
\label{sect_examples}
In this section we provide two examples of minimal reversible nonsymmetric rings, one which is duo and one that is not. The first example is a byproduct of the proof to Proposition \ref{prop_23}.

\begin{example}
\label{ex_232}
Let
\[
R=\frac{\F_2\id{u,v}}{\id{u^3,v^3,u^2+v^2+vu,vu^2+uvu+vuv}}.
\]
We will show that this is a minimal reversible nonsymmetric ring that is neither left nor right duo.

First, $R$ is an 8-dimensional algebra over $R/J(R)\cong\F_2$ with basis
\[
\{1,v_1=u,v_2=v,v_3=u^2,v_4=uv,v_5=v^2,v_6=uvu,v_7=vuv\}
 \]
showing $|R|=256$. Second, since
\[
uv(u+v)=u^3+uv(u+v)=u^3+uvu+uv^2=u(u^2+v^2+vu)=0
\]
and
\[
vu(u+v)=vu^2+vuv=uvu\neq 0,
\]
$R$ is nonsymmetric. Third, since $uv\notin vR$ and $uv\notin Ru$, $R$ is neither right nor left duo. Finally, we show $R$ is reversible. Let $a,b\in J(R)$ with $a=\sum_{i=1}^7\alpha_iv_i$ and $b=\sum_{i=1}^7\beta_iv_i$ where $\alpha_i,\beta_i\in R/J(R)$. Then
\begin{eqnarray*}
ab&=&(\alpha_1\beta_1+\alpha_2\beta_1)u^2+(\alpha_1\beta_2)uv+(\alpha_2\beta_2+\alpha_2\beta_1)v^2+\\
&&(\alpha_1\beta_4+\alpha_2\beta_3+\alpha_2\beta_4+\alpha_5\beta_1+\alpha_3\beta_2)vuv+\\
&&(\alpha_1\beta_5+\alpha_2\beta_3+\alpha_4\beta_1+\alpha_5\beta_1+\alpha_4\beta_2)uvu.
\end{eqnarray*}
Assume $ab=0$. From the coefficient of $uv$, $\alpha_1=0$ or $\beta_2=0$. Then, considering the coefficients of $u^2$ and $v^2$ there are three restrictions:
(1) $\alpha_2=0$ or $\beta_1=0$, (2) $\alpha_1=0$ or $\beta_1=0$ and (3) $\alpha_2=0$ or $\beta_2=0$. This produces 7 possibilities which we list along with their implications.
\[
\begin{array}{|c|c|c|c|l|l|}
\hline
\alpha_1&\beta_2&\alpha_2&\beta_1&&\\\hline
0&0&0&0&ab=0&\\
0&0&0&1&ab=\alpha_5vuv+(\alpha_4+\alpha_5)uvu&\alpha_4=\alpha_5=0\\
0&0&1&0&ab=(\beta_3+\beta_4)vuv+\beta_3uvu&\beta_3=\beta_4=0\\
0&1&0&0&ab=\alpha_3vuv+\alpha_4uvu&\alpha_3=\alpha_4=0\\
0&1&0&1&ab=(\alpha_3+\alpha_5)vuv+\alpha_5uvu&\alpha_3=\alpha_5=0\\
1&0&0&0&ab=\beta_4vuv+\beta_5uvu&\beta_4=\beta_5=0\\
1&0&1&0&ab=\beta_3vuv+(\beta_3+\beta_5)uvu&\beta_3=\beta_5=0\\
  \hline
\end{array}
\]
Then for any of these situations,
\begin{eqnarray*}
ba&=&(\beta_1\alpha_1+\beta_2\alpha_1)u^2+(\beta_1\alpha_2)uv+(\beta_2\alpha_2+\beta_2\alpha_1)v^2+\\
&&(\beta_1\alpha_4+\beta_2\alpha_3+\beta_2\alpha_4+\beta_5\alpha_1+\beta_3\alpha_2)vuv+\\
&&(\beta_1\alpha_5+\beta_2\alpha_3+\beta_4\alpha_1+\beta_5\alpha_1+\beta_4\alpha_2)uvu=0.
\end{eqnarray*}
Hence, $R$ is reversible.
\end{example}

As with the previous example, the next example, which was originally given in \cite{marks_2002}, may have been discovered from the proof of Proposition \ref{prop_2211}.

\begin{example}
\label{ex_FQ}
Let $Q_8=\{\pm1,\pm i,\pm j,\pm k\}$ be the quaternions, $\{x_g: g\in Q_8\}$ and $R=\F_2Q_8$ where we write elements of $R$ as $\F_2$-linear combinations of $\{x_g: g\in Q_8\}$. Let $u=1+x_i$ and $v=1+x_j$. Since $uv(uv)\neq 0$ and $vu(uv)=0$, $R$ is nonsymmetric. We first verify these key relations, $u^4=0$, $v^4=0$, $u^2+v^2=0$, and $u^2+uv+vu+uvu=0$. Then one can see that $J(R)=\F_2u+\F_2v+J(R)^2$, $J(R)^2=\F_2uv+\F_2vu+J(R)^3$, $J(R)^3=\F_2uv^2+\F_2vu^2+J(R)^4$ and $J(R)^4=\F_2uv^3$. Then we deduce that
\[
R\cong\frac{\F_2\id{u,v}}{\id{u^4,v^4,u^2+v^2,u^2+uv+vu+uvu}}.
\]
These facts fit with the proof of Proposition \ref{prop_2211} if the restriction of $J(R)^3$ being principal is removed. Had we found this ring we could have verified reversibility similar to what was done for Example \ref{ex_232} although the proof in \cite{marks_2002} is more elegant. Though their proof is predicated on $R$ being right duo for which their proof is not straight forward. From our view of $R$ we deduce the duo property very simply. We need to show that $sr\in rR$ for any $r,s\in R$. From the generators of $J(R)^i$, we see we need only verify that $uv\in vR$ and that  $vu\in uR$. First, note that $vuv=uvu$. Then we have that
\[
vu=u^2+uv+uvu=u(u+v+vu)
\]
and
\[
uv=u^2+vu+uvu=v^2+vu+vuv=v(v+u+uv).
\]
\end{example}

To close we point out that we have two examples of minimal reversible nonsymmetric rings, one which is duo and the other which is not. Furthermore, one has a jacobson radical of nilpotency of 4 and the other of 5. From Propositions \ref{prop_j3} and \ref{prop_21s} we see that the nilpotency of jacobson radical must be between 4 and 6. Can we find a minimal reversible nonsymmetric ring that has a jacobson radical of nilpotency 6? From the results in Section \ref{sect_main} we know that $D=(2,2,1,1,1)$ (as defined in Section \ref{sect_main}) if such an example exists. Also, in viewing the proof to Proposition \ref{prop_321}, is there a minimal reversible nonsymmetric ring with $D=(3,2,2)$. Finally, can we identify all the minimal reversible nonsymmetric rings?

\bibliographystyle{plain}
\bibliography{../SteveSzaborefs}

\begin{thebibliography}{10}

\bibitem{cohn_1999}
P.~M. Cohn.
\newblock Reversible rings.
\newblock {\em Bull. London Math. Soc.}, 31(6):641--648, 1999.

\bibitem{corbas_2000}
B.~Corbas and G.~D. Williams.
\newblock Rings of order {$p^5$}. {I}. {N}onlocal rings.
\newblock {\em J. Algebra}, 231(2):677--690, 2000.

\bibitem{corbas_2000_2}
B.~Corbas and G.~D. Williams.
\newblock Rings of order {$p^5$}. {II}. {L}ocal rings.
\newblock {\em J. Algebra}, 231(2):691--704, 2000.

\bibitem{derr_1994}
J.~B. Derr, G.~F. Orr, and Paul~S. Peck.
\newblock Noncommutative rings of order {$p^4$}.
\newblock {\em J. Pure Appl. Algebra}, 97(2):109--116, 1994.

\bibitem{gutan_2004}
Marin Gutan and Andrzej Kisielewicz.
\newblock Reversible group rings.
\newblock {\em J. Algebra}, 279(1):280--291, 2004.

\bibitem{jacobson_1964}
Nathan Jacobson.
\newblock {\em Structure of rings}.
\newblock American Mathematical Society Colloquium Publications, Vol. 37.
  Revised edition. American Mathematical Society, Providence, R.I., 1964.

\bibitem{kim_2011}
Byung-Ok Kim and Yang Lee.
\newblock Minimal noncommutative reversible and reflexive rings.
\newblock {\em Bull. Korean Math. Soc.}, 48(3):611--616, 2011.

\bibitem{kim_2003}
Nam~Kyun Kim and Yang Lee.
\newblock Extensions of reversible rings.
\newblock {\em J. Pure Appl. Algebra}, 185(1-3):207--223, 2003.

\bibitem{lambek_1971}
J.~Lambek.
\newblock On the representation of modules by sheaves of factor modules.
\newblock {\em Canad. Math. Bull.}, 14:359--368, 1971.

\bibitem{marks_2002}
Greg Marks.
\newblock Reversible and symmetric rings.
\newblock {\em J. Pure Appl. Algebra}, 174(3):311--318, 2002.

\bibitem{marks_2003}
Greg Marks.
\newblock A taxonomy of 2-primal rings.
\newblock {\em J. Algebra}, 266(2):494--520, 2003.

\bibitem{szabo_2017_2}
Steve Szabo.
\newblock Some minimal rings related to 2-primal rings.
\newblock {\em in preparation}, 2017.

\bibitem{xu_1998}
Liqiong Xu and Weimin Xue.
\newblock Structure of minimal non-commutative zero-insertive rings.
\newblock {\em Math. J. Okayama Univ.}, 40:69--76 (2000), 1998.

\bibitem{xue_1992}
Weimin Xue.
\newblock Structure of minimal noncommutative duo rings and minimal strongly
  bounded nonduo rings.
\newblock {\em Comm. Algebra}, 20(9):2777--2788, 1992.

\end{thebibliography}

\end{document}